\documentclass[11pt,letterpaper]{article}
\usepackage{tikz-cd}
\usepackage{bbm}
\usepackage{graphicx}
\usepackage{galois}
\usepackage{stmaryrd}
\usepackage{mathpazo}
\usepackage{amsmath}
\usepackage{amsthm}
\usepackage{amssymb}

\theoremstyle{plain}
\newtheorem{theorem}{Theorem}[section]
\newtheorem{lemma}{Lemma}[section]
\newtheorem{proposition}{Proposition}[section]
\newtheorem{corollary}{Corollary}[section]
\theoremstyle{definition}
\newtheorem{definition}{Definition}[section]
\newtheorem{example}{Example}[section]
\newcommand{\cat}[1]{{\mathbf{\mathsf{{#1}}}}}

\newcommand{\cok}{\mathrm{cok\;}}

\newcommand{\Cok}{\mathrm{Cok\;}}
\newcommand{\Ker}{\mathrm{Ker\;}}
\newcommand{\Nim}{\mathrm{Nim\;}}

\renewcommand{\Im}{\mathrm{Im\;}}

\renewcommand{\hom}{\cat{Hom}}

\newcommand{\setof}[1]{\left\{ {#1} \right\}}

\newcommand{\define}[1]{\textbf{#1}}

\newcommand{\id}{\text{id}}

\renewcommand{\epsilon}{\varepsilon}
\renewcommand{\phi}{\varphi}

\renewcommand{\mapsto}{\rightsquigarrow}
\newcommand{\meet}{\wedge}
\newcommand{\join}{\vee}
\newcommand{\bigmeet}{\bigwedge}
\newcommand{\bigjoin}{\bigvee}

\newcommand{\nats}{\mathbb{N}}

\renewcommand{\restriction}[2]{{% we make the whole thing an ordinary symbol
  \left.\kern-\nulldelimiterspace % automatically resize the bar with \right
  #1 % the function
  \vphantom{|} % pretend it's a little taller at normal size
  \right|_{#2} % this is the delimiter
  }}

\newcommand{\sbt}{\,\begin{picture}(1,1)(-1,-2)\circle*{2}\end{picture}\ }
\newcommand{\fc}{\trianglelefteqslant}
\newcommand{\fcop}{\trianglerighteqslant}
\newcommand{\resmap}[3]{\mathcal{#1}_{#2 \fc #3}}
\newcommand{\resmaplower}[3]{(\mathcal{#1}_{\!#2 \fc #3})_{\sbt}}
\newcommand{\resmapupper}[3]{(\mathcal{#1}_{#2 \fc #3})^{\sbt}}

\newcommand{\pwrset}[1]{2^{#1}}
\newcommand{\powerset}[1]{\pwrset{#1}}

\newcommand{\covers}{\sqsubset}
\newcommand{\lattice}[1]{\mathbf{#1}}
\newcommand{\poset}[1]{\mathrm{#1}}
\newcommand{\height}{h}
\newcommand{\upset}[1]{\uparrow\!{#1}}
\newcommand{\downset}[1]{\downarrow\!{#1}}

\newcommand{\fixedpoint}{\mathrm{Fix}}
\newcommand{\Space}[1]{{\mathrm{#1}}}  % MAYBE \mathcal?  I'M FRUSTRATED WITH THIS FONT
\newcommand{\Laplacian}{L}
\newcommand{\boundary}{\partial}
\newcommand{\coboundary}{\delta}
\newcommand{\Flow}{\Phi}
\newcommand{\TarH}{T\!H} % tarski cohomology
\newcommand{\GraH}{G\!H} % grandis cohomology
\DeclareMathOperator*{\limit}{lim}

\newcommand{\diagonal}{\Delta}
\newcommand{\codiagonal}{\nabla}
\newcommand{\qcoboundary}{\tilde{\coboundary}}

\newcommand{\HodHup}{{}_{+}H\!H}
\newcommand{\HodHdown}{{}_{-}H\!H}
\usepackage{extarrows}

\usepackage{xpatch}
\xpretocmd{\footnote}{\unskip}{}{}

\newcommand{\sheaf}[1]{\underline{\mathcal{#1}}}

\begin{document}
\author{Robert Ghrist\thanks{
	Department of Mathematics, Department of Electrical \& Systems Engineering \newline
	\hspace*{1.5em} University of Pennsylvania \newline
	\hspace*{1.5em} \texttt{ghrist@math.upenn.edu}
%	\hspace*{1.5em} David Rittenhouse Lab \newline
%	\hspace*{1.5em} 209 S.~33rd Street \newline
%	\hspace*{1.5em} Philadelphia, PA 19104
} \and Hans Riess\thanks{
	Department of Electrical \& Systems Engineering \newline
	\hspace*{1.5em} University of Pennsylvania \newline
	\hspace*{1.5em} \texttt{hmr@seas.upenn.edu}
%	\hspace*{1.5em} Levine Hall \newline
%	\hspace*{1.5em} 3330 Walnut Street \newline
%	\hspace*{1.5em} Philadelphia, PA 19104
}}	
\date{}	
%@@@@@@@@@@@@@@@@@@@@@@@@@@@@@@@@@@@@@@@@@@@@@@@@@@@@@@@@@@@@@@@@@@@@@@@@@@@@@@@@@@@@@@@@@@@@@@@@@@@@@@@@@@@@@	
%\author{Robert Ghrist}
%\email{ghrist@math.upenn.edu}
%\address{University of Pennsylvania \\
%Departments of Mathematics\\
%Department of Electrical \& Systems Engineering \\
%David Rittenhouse Lab \\
%209 South 33rd Street \\
%Philadelphia, PA 19104}

%\author{Hans Riess}
%\email{hmr@seas.upenn.edu}
%\address{University of Pennsylvania \\
%Department of Electrical \& Systems Engineering \\
%Levine Hall \\
%3330 Walnut Street \\
%Philadelphia, PA 19104}
%
%
%%\thanks{The authors were supported by Office of Naval Research (Grant
%%No. N00014-1442-16-1-2010).}
%\classification[2020]{18B35, 18F20, 18G90, 55N30, 55N31, 05C50.}
%\keywords{Cellular sheaves, lattice theory, non-abelian homological algebra}
%
\title{Cellular Sheaves of Lattices and the Tarski Laplacian}
%
%\received{Month XX, XXXX}   % receive date (for example: October 11, 1999)
%\revised{Month XX, XXXX}    % date of revision; omit, if no revision;
%                             % if multiple revisions, separate by commas
%%\accepted{Month Day, Year}  % acceptance date
%\published{Month XX, XXXX}  % publish date
%\submitted{Peter Bubenik}      % Name of Journal's Editor, who handled Article 
%\volumeyear{2020} % Volume Year
%\volumenumber{11} % Volume Number 
%\issuenumber{1}   % Issue Number
%\startpage{1}     % PageNumber of first page
%\articlenumber{1} % Sequence number of article within issue

% If copyright is retained by author, comment this out:
%\owner{International Press}

\maketitle

\begin{abstract}
	This paper initiates a discrete Hodge theory for cellular sheaves taking values in a category of lattices and Galois connections. The key development is the \emph{Tarski Laplacian}, an endomorphism on the cochain complex whose fixed points yield a cohomology that agrees with the global section functor in degree zero. This has immediate applications in consensus and distributed optimization problems over networks and broader potential applications.
\end{abstract}

% @@@@@@@@@@@@@@@@@@@@@@@@@@@@@@@@@@@@@@@@@@@@@@@@@@@@@@@@@@@@@@@@@@@@@@@@@@@@@@@@@@@@@@@@@@@@@@@@@@@@@@@@@@@@@@@@@
\section{Introduction}
\label{sec:introduction}
% @@@@@@@@@@@@@@@@@@@@@@@@@@@@@@@@@@@@@@@@@@@@@@@@@@@@@@@@@@@@@@@@@@@@@@@@@@@@@@@@@@@@@@@@@@@@@@@@@@@@@@@@@@@@@@@@@

The goal of this paper is to initiate a theory of sheaf cohomology for cellular sheaves valued in a category of lattices. Lattices are algebraic structures with a rich history \cite{rota1997many} and a wide array of applications \cite{denning1976lattice,barthelemy1991formal,farley2003breaking,sandhu1993lattice,monjardet2004lattices,erdmann2017topology}. Cellular sheaves are data structures that stitch together algebraic entities according to the pattern of a cell complex \cite{shepard1985cellular}. Sheaf cohomology is a compression that collapses all the data over a topological space --- or cell complex --- to a minimal collection that intertwines with the homological features of the base space \cite{iversen2012cohomology}. 

% ---------------------------------------------------------------------------------------------------
\subsection{Contributions}
\label{sec:contribs}
% ---------------------------------------------------------------------------------------------------

Our approach is to set up a Hodge-style theory, developing analogues of the combinatorial Laplacian adapted to sheaves of lattices. Specific contributions of this work include the following.
\begin{enumerate}
  \item In \S\ref{sec:background}, we review posets, lattices, lattice connections, cellular sheaves, and Hodge theory. It is nontrivial to define cohomology for sheaves valued in the (nonabelian) category of lattices and connections. 
  \item In \S\ref{sec:defnprop}, we define an endomorphism on cochains of a cellular sheaf of lattices and begin arguing that this \define{Tarski Laplacian}, $\Laplacian$, is a reasonable candidate for a diffusion operator. 
  \item In \S\ref{sec:FPT}, we prove the main result that $(\id\meet\Laplacian)$ has fixed point set equal to the quasi-sublattice of global sections of the sheaf. 
  \item In \S\ref{sec:diffusion}, we show when the resulting discrete-time \define{harmonic flow} projects arbitrary 0-cochains to global sections.
  \item Interpreting global sections as zeroth cohomology of the sheaf, in \S\ref{sec:TC} we define higher \define{Tarski cohomology} in terms of fixed points on higher cochains. 
  \item In \S\ref{sec:grandis}, we compare and contrast this cohomology with that implicit in the works of Grandis in the case of sheaves of lattices that factor through sheaves of vector spaces.  
  \item Finally, in \S\ref{sec:HC}, we attempt to build a Hodge Laplacian and cohomology theory directly from a (pseudo) cochain complex, comparing and contrasting with the Tarski theory. 
\end{enumerate}
The results of these inquiries are summarized in Table \ref{table:1} of \S\ref{sec:summary}.

% ---------------------------------------------------------------------------------------------------
\subsection{Motivations}
\label{sec:motives}
% ---------------------------------------------------------------------------------------------------

Readers who need no motivation for Hodge-theoretic sheaf cohomology valued in lattices may skip this subsection. 

The authors are motivated by certain problems in data science of a local-to-global nature, in which multiple instances of local data are required to satisfy constraints based on a notion of proximity. Such problems often can be formalized in the language of sheaf theory, where the local data are stored in stalks, and local constraints are encoded in restriction maps. The cohomology of a sheaf collates global information, such as global sections and obstructions to such. 

Certain sheaves prominent in applications take values in sets. For example, Reeb graphs can be viewed in terms of global sections of (co)sheaves valued in finite sets \cite{de2016categorified, curry2018fiber}. Other examples arising from problems in quantum computation can be found in the works of Abramsky et al. \cite{abramsky2011sheaf}. Applications closer to engineering are present in the pioneering work of Goguen \cite{goguen1992sheaf} and in more recent work of others \cite{robinson2014topological, robinson2017sheaves}. The lack of a full sheaf cohomology theory in these settings limits applicability. 

Recently, there has been substantial activity in cellular sheaves (and dual cosheaves), prompted by the thesis of Curry \cite{curry2014sheaves}. The theory of cellular sheaves valued in vector spaces is especially well-developed, and their cohomology is not difficult to define or compute \cite{ghrist2014elementary,curry2016discrete}. Cellular sheaves of vector spaces and their cohomology have been used in network flow and coding problems \cite{coding2011ghrist}, persistent homology computation \cite{ghrist2020persistence}, signal processing \cite{robinson2013understanding}, distributed optimization \cite{hansen2019distributed}, opinion dynamics \cite{hansen2020opinion}, and more. 

Our motivations for working with sheaves of lattices stems from both their generality (ranging from lattices of subgroups to Boolean algebras) and broad applicability in logic, topology, and discrete mathematics. The reader need look no further than computational topology for recent work harnessing lattice theory (e.g.\ in computational Conley theory \cite{harker2018computational}, classifying embeddings factoring through a Morse function \cite{catanzaro2020moduli}, and computing interleavings of generalized persistence modules \cite{botnan2020relative}). The desire for a Hodge theory comes from more than mere computation of sheaf cohomology. For a sheaf of vector spaces, the Hodge Laplacian has numerous applications, as detailed in the thesis of Hansen \cite{hansen2020laplacians}. As a generalization of the graph Laplacian, the Hodge Laplacian for sheaves of vector spaces has a rich spectral theory \cite{hansen2019toward} and leads to notions of harmonic extension that are useful in several contexts \cite{hansen2020opinion, hansen2020laplacians}. 

Among the many contingent avenues for applications, one is of special note. Graph signal processing --- the extension of signal processing methods from signals over the reals to signals over graphs --- has of late been a vibrant hive of activity, all based on the graph Laplacian, and all amenable to the Hodge-theoretic approach to sheaves over networks. Lifting graph signal processing from real-valued to lattice-valued data is an intriguing concept, with only the first steps being imagined by P{\"u}schel et al,~for functionals on semilattices \cite{puschel2019discrete} and powersets \cite{puschel2018discrete}. This paper provides the technical background for establishing graph signal processing valued in lattices. Other potentialities, especially those concerning deep learning and convolutional neural nets wait in the wings. 

% @@@@@@@@@@@@@@@@@@@@@@@@@@@@@@@@@@@@@@@@@@@@@@@@@@@@@@@@@@@@@@@@@@@@@@@@@@@@@@@@@@@@@@@@@@@@@@@@@@@@@@@@@@@@@@@@@
\section*{Acknowledgments}
S.\ Krishnan initially suggested to the authors a fixed-point type theorem for sheaf cohomology.
Conversations with G.\ Henselman-Petrusek, J. Hansen, K.\ Spendlove, and J.-P.\ Vigneaux were helpful. We would also like to thank the reviewers for their insightful comments.
This work was funded by the Office of the Assistant Secretary of Defense Research \& Engineering through a Vannevar Bush Faculty Fellowship, ONR N00014-16-1-2010.
% @@@@@@@@@@@@@@@@@@@@@@@@@@@@@@@@@@@@@@@@@@@@@@@@@@@@@@@@@@@@@@@@@@@@@@@@@@@@@@@@@@@@@@@@@@@@@@@@@@@@@@@@@@@@@@@@@

% @@@@@@@@@@@@@@@@@@@@@@@@@@@@@@@@@@@@@@@@@@@@@@@@@@@@@@@@@@@@@@@@@@@@@@@@@@@@@@@@@@@@@@@@@@@@@@@@@@@@@@@@@@@@@@@@@
\section{Background}
\label{sec:background}
% @@@@@@@@@@@@@@@@@@@@@@@@@@@@@@@@@@@@@@@@@@@@@@@@@@@@@@@@@@@@@@@@@@@@@@@@@@@@@@@@@@@@@@@@@@@@@@@@@@@@@@@@@@@@@@@@@

The following is terse but sufficient for the remainder of the paper. The reader may find more detailed references for lattice theory \cite{davey2002introduction,roman2008lattices,gierz2012compendium}, cellular sheaves \cite{curry2014sheaves}, and their Hodge theory \cite{hansen2019toward}. There is some variation in terminology among references in lattice theory. {\em Caveat lector.}

% ---------------------------------------------------------------------------------------------------
\subsection{Posets and Lattices}
\label{sec:lattices}
% ---------------------------------------------------------------------------------------------------

A partially ordered set, or \define{poset}, is a set $\poset{P}$ with a binary relation, $\preceq$, satisfying \emph{reflexivity}, $x \preceq x$, \emph{transitivity}: $x \preceq y$ and $y \preceq z$ implies $x \preceq z$, and \emph{anti-symmetry}: $x \preceq y$ and $y \preceq x$ implies $x = y$. Denote a \emph{strict} partial order, $x \prec y$, if $x \preceq y$, but $y \npreceq x$.

Given two elements $x$ and $y$ of a poset $\poset{P}$, define the \define{meet}, $x \meet y$, and  \define{join}, $x \join y$, to be the \emph{greatest lower bound} and \emph{least upper bound} respectively. That is,
\[
    x \meet y = \max \setof{z: z \preceq x, z \preceq y}
  \quad : \quad  
     x \join y = \min \setof{z: z \succeq x, z \succeq y}.
\]
Likewise, for any subset $S \subseteq \poset{P}$, we may define $\bigwedge S$ and $\bigvee S$ whenever they exist.

A \define{lattice} is a poset $\lattice{X}$ closed under all finite (possibly empty) meets and joins. A lattice is \define{complete} if all arbitrary meets and joins exist (finite lattices thus being complete). By this definition, $0 = \join~\emptyset$ and $1 = \meet~\emptyset$ exist, so that all lattices have maximal ($1$) and minimal ($0$) elements. Such lattices are sometimes called \emph{bounded lattices} in the literature: our convention is such that all lattices in this paper are bounded.

An element $x \in \lattice{X}$ is \define{join irreducible} if $x \neq 0$ and $x = y \join z$ implies $x = y$ or $x = z$.
An element $x \in \lattice{X}$ is \define{meet irreducible} if $x \neq 1$ and $x = y \meet z$ implies $x = y$ or $x = z$.
A lattice can be defined either algebraically or by its partial order. From the binary operations join and meet, we can recover the partial order by $x \preceq y ~ \Leftrightarrow ~ x \join y = y ~ \Leftrightarrow ~ x \preceq y \Leftrightarrow x \meet y = x$.

For $x \in \poset{P}$, a poset, define the \define{principal downset} of x, $\downset{x} =   \{y \in P: \\ y \preceq x \}$, and the \define{principal upset} of $x$, $\upset{x} = \setof{y \in L:~y \succeq x}$. A subset $D \subseteq \poset{P}$ is a \define{downset} if for all $x \in D$, if $y \preceq x$, then $y \in D$. A subset $U \subseteq \poset{P}$ is an \define{upset} if for all $x \in U$, if $y \succeq x$, then $y \in U$. 

An \define{interval} in a poset is a set, $\left[ x, y \right] = \: \downset{y} \: \cap \: \upset{x}$. A \define{sublattice} $\lattice{S} \subseteq \lattice{X}$ is a subset closed under meets and joins, including $0$ and $1$. It is more common to use a weaker notion of a sublattice. Define a \define{quasi-sublattice} $\lattice{Q} \subseteq \lattice{X}$ to be a subset that is a lattice in its own right, in particular, not necessarily containing $0$ or $1$. The product of a family of lattices $\setof{\lattice{X}_\alpha}_{\alpha \in J}$ is the cartesian product $\prod_{\alpha \in J} \lattice{X}_\alpha$ with meets and joins defined coordinate-wise. We denote an element of the product $\mathbf{x} \in\prod_{\alpha \in J} \lattice{X}_\alpha$ and $\mathbf{0}$ and $\mathbf{1}$ the bottom and top elements of the product lattice.

In a poset $\poset{P}$, we say $y$ \define{covers} $x$, denoted $x \covers y$, if $x \prec y$ and there does not exist $z \in \poset{P}$ such that $x \prec z \prec y$. A lattice (or poset) $\lattice{X}$ is \define{graded} if there exist a ranking $r: \lattice{X} \to \nats$ such that $r(x) < r(y)$ whenever $x \prec y$ and $r(y) = r(x)+1$ whenever $x \covers y$. For example, given a vector space $V$, let $\cat{Gr}(V)$ be the (Grassmanian) poset of subspaces of $V$ with the order $U \preceq U'$ if $U$ is a subspace of $U'$. One may check that this is a lattice with $U \join W = U + W$ the subspace sum, and $U \meet W = U \cap W$ the intersection. For $V$ finite dimensional, $\cat{Gr}(V)$ is graded via dimension: $r(U) = \dim(U)$. For another example, let $\powerset{S}$ be the powerset lattice of a set, $S$, ordered by inclusion. For $S$ finite, this lattice is graded by cardinality: $r(U \subseteq S) = \#U$.

A subset $\lattice{I}\subseteq \poset{P}$ is a \define{chain} if $\lattice{I}$ is totally-ordered. A finite chain $\lattice{I} = \setof{x_0 \prec x_1 \prec \cdots \prec x_\ell}$ is said to have \define{length} $\ell$. A poset $\poset{P}$ satisfies the \define{descending chain condition} (DCC) if no infinite strictly descending chain exists. 
\begin{lemma}
\label{thm:graded-dcc}
Any graded poset satisfies the descending chain condition.
\end{lemma}
\begin{proof}
The grading of any strictly descending chain is a strictly decreasing sequence of natural numbers and thus finite. 
\end{proof}

The \define{height} of a poset, denoted $\height(\poset{P})$, is the length of the maximal chain if it exists; otherwise, $\height(\poset{P}) = \infty$. The \define{distance} $d_\poset{P}(x,y)$ between $x, y \in \poset{P}$ is defined as the interval height $\height\left( \left[x, y \right] \right)$. Height is additive under (finite) products.
\begin{lemma}
\label{thm:height-product}
For $\setof{\poset{P}_i}_{i=1}^{N}$ be a finite family of posets,
\begin{equation}
\height \left( \prod_{i=1}^N \poset{P}_i \right) = \sum_{i=1}^N \height(\poset{P}_i)
\end{equation}
\end{lemma}
\begin{proof}
It suffices by finite induction to show $\height(\poset{P} \times \poset{P}') = \height(\poset{P}) + \height(\poset{P}')$. A maximal chain of $\poset{P} \times \poset{P}'$ projects to maximal chains $\lattice{I}$ and $\lattice{I}'$ in $\poset{P}$ and $\poset{P}'$ respectively and lies in the product $\lattice{I} \times \lattice{I}'$. By projection and maximality, its length is $\ell(\lattice{I}) + \ell(\lattice{I}')$.
\end{proof}

% ---------------------------------------------------------------------------------------------------
\subsection{The Tarski Fixed Point Theorem}
\label{sec:TFPT}
% ---------------------------------------------------------------------------------------------------

We are especially concerned with maps of lattices $f: \lattice{X} \rightarrow \lattice{Y}$. A (poset or) lattice map $f$ is \define{order-preserving} if $x \preceq x'$ implies $f(x) \preceq f(x')$. A lattice map $f$ is \define{join-preserving} if $f(x \join x') = f(x) \join f(x')$ and $f(0) = 0$. Similarly, $f$ is \define{meet-preserving} if $f(x \meet x') = f(x) \meet f(x')$ and $f(1) = 1$. If $f: \lattice{X} \rightarrow \lattice{Y}$ is join preserving, then it is automatically order-preserving since, if $x \preceq x'$, then $x \join x' = x'$ and
\[
  f(x) \join f(x') = f( x \join x') = f(x')
\]
which holds if and only if $f(x) \preceq f(x')$. The same holds for meet-preserving maps.

An order-preserving map $\Phi: \lattice{X} \rightarrow \lattice{X}$ is \define{expanding} if $\Phi(x) \succeq x$ and \define{contracting} if $\Phi(x) \preceq x$. The \define{fixed point set} of $\Phi: \lattice{X} \rightarrow \lattice{X}$ is the set $\fixedpoint(\Phi) = \setof{x \in \lattice{X}~:~\Phi(x) = x}$. The fixed point set can be realized as the intersection of the \define{prefix points} of $\Phi$, $\mathrm{Pre}(\Phi) = \setof{x \in X:~\Phi(x) \preceq x}$ and the \define{suffix points} of $\Phi$, $\mathrm{Post}(\Phi) = \setof{x \in X:~\Phi(x) \succeq x}$. 

The critical result about the fixed point set of a lattice endomorphism concerns its subobject structure within the lattice.
%The classical \define{Tarski Fixed Point Theorem} is well-known \cite{roman2008lattices}, but, for completeness and for use of perspectives in the sequel, we give a brief proof. 

\begin{theorem}[Tarski Fixed Point Theorem]
\label{thm:tarski}
The fixed point set of an order-preserving endomorphism of a complete lattice is a complete lattice.
\end{theorem}
%
%\begin{proof} 
%Let $\Phi:\lattice{X}\to\lattice{X}$ be order-preserving for $\lattice{X}$ complete. For $S\subset\mathrm{Pre}(\Phi)$ nonempty and $x\in S$,
%\[
%  \Phi\left( \bigmeet S \right) \preceq \Phi(x) \preceq x .
%\] 
%As this holds for every $x\in S$, $\bigmeet S \in \mathrm{Pre}(\Phi)$. Since 
%\[
%  \bigjoin S = \bigmeet\left( \bigcap_{x \in S} \upset{x} \right) ,
%\]
%it follows that $\mathrm{Pre}(\Phi)$ is complete. By duality, $\mathrm{Post}(\Phi)$ is complete also. The fixed point set of $\Phi$ can be expressed as 
%\[
%  \fixedpoint(\Phi) 
%  = 
%  \mathrm{Post}
%  \left(
%    \restriction{\Phi}{\mathrm{Pre}(\Phi)}
%  \right) .
%\]
%Hence, $\fixedpoint(\Phi)$ is complete.
%\end{proof}

The computational complexity of finding a fixed point via queries to $\Phi$ is well-understood 
\cite{chang2008complexity}.

% ---------------------------------------------------------------------------------------------------
\subsection{Galois Connections}
\label{sec:connections}
% ---------------------------------------------------------------------------------------------------

There is a type of lattice map which interfaces well with categorical methods and homological algebra. A \define{(Galois) connection} between a pair of lattices $(\lattice{X}, \lattice{Y})$ is an order-preserving pair, $f = (f_{\sbt}, f^{\sbt})$,
\begin{equation}
    \lattice{X} \galois{f_{\sbt}}{f^{\sbt}} \lattice{Y} 
    \quad {\rm such~that} \quad
    f_{\sbt}(x) \preceq y \Leftrightarrow x \preceq f^{\sbt}(y) \quad {\rm for~all} \quad x \in \lattice{X},~y \in \lattice{Y}.
\end{equation}
One calls $f_{\sbt}$ the \textit{lower} or \textit{left adjoint}, and $f^{\sbt}$ the \textit{upper} or \textit{right adjoint}. One interpretation of a connection is a best approximation to an order inverse, as the following proposition indicates.
\begin{proposition}
\label{thm:connections}
The following are equivalent:
\begin{enumerate}
    \item $f_{\sbt}(x) \preceq y \Leftrightarrow x \preceq f^{\sbt}(y)$ for all $x \in \lattice{X}, y \in \lattice{Y}$;
    \item $f^{\sbt} f_{\sbt} \succeq \id_X$ and $f_{\sbt} f^{\sbt} \preceq \id_Y$.
\end{enumerate}
\end{proposition}
\begin{proof}
Suppose $f_{\sbt} (x) \preceq y$ if and only if $x \preceq f^{\sbt} (y)$ for all $x \in \lattice{X}$, $y \in \lattice{Y}$.
In particular, by reflexivity, $f_{\sbt} (x) \preceq f_{\sbt}(x)$.
Hence, $f^{\sbt} f_{\sbt} (x) \succeq x$ for all $x \in \lattice{X}$.
Similarly, by reflexivity, $f^{\sbt} (y) \preceq f^{\sbt} (y)$.
Hence, $f_{\sbt} f^{\sbt} (y) \preceq y$ for all $y \in \lattice{Y}$. 
\par
Conversely, suppose $f^{\sbt} f_{\sbt}(x) \succeq x$ and $f_{\sbt} f^{\sbt}(y) \preceq y$ for all $x \in \lattice{X}$, $y \in \lattice{Y}$.
If $f_{\sbt} (x) \preceq y$, then, since $f^{\sbt}$ is order-preserving, $f^{\sbt} f_{\sbt} (x) \preceq f^{\sbt} (y)$.
But, $f^{\sbt} f_{\sbt}(x) \succeq x$ which implies $x \preceq f^{\sbt} (y)$ by transitivity.
Similarly, if $x \preceq f^{\sbt} (y)$, then, since $f_{\sbt}$ is order-preserving, $f_{\sbt} (x) \preceq f_{\sbt} f^{\sbt} (y)$ which implies $f_{\sbt} (x) \preceq y$ since $f_{\sbt} f^{\sbt}(y) \preceq y$.
\end{proof}

A category-theoretic interpretation is in order. Viewing a lattice $\lattice{X}$ as a category with the underlying poset structure defining morphisms, the meet and join operations are the product and coproduct respectively, $0$ is the initial object, and $1$ is the terminal object. Functors between lattices are precisely order-preserving maps and a connection between lattices is precisely an adjunction. This leads to the following well-known useful characterization of connections.

\begin{proposition}
\label{thm:aft}
If $f = (f_{\sbt}, f^{\sbt})$, then $f_{\sbt}$ preserves joins and $f^{\sbt}$ preserves meets. Conversely, if $f_{\sbt}: \lattice{X} \rightarrow \lattice{Y}$ preserves joins, then there exist a map $f^{\sbt}: \lattice{Y} \rightarrow \lattice{X}$ that preserves meets such that $(f_{\sbt}, f^{\sbt})$ is a connection. Dually, if $f^{\sbt}: \lattice{Y} \rightarrow \lattice{X}$ preserves meets, then there exist a map $f_{\sbt}: \lattice{X} \rightarrow \lattice{Y}$ that preserves joins such that $(f_{\sbt}, f^{\sbt})$ is a connection.
Explicitly, these are given as
\begin{equation}
\label{eq:adjoint-formula}
    f^{\sbt}(y) = \bigvee f_{\sbt}^{-1}(\downset{y})
  \quad
  {\rm and }
  \quad
  f_{\sbt}(x) = \bigwedge f^{\sbt -1} (\upset{x}) .
\end{equation}
\end{proposition}
\begin{proof}
The first statement follows from the (co)limit-preserving properties of adjuctions. The second statement is a consequence of the Adjoint Functor Theorem. For completeness, we give a constructive proof in one direction (the other following from duality). Given $f_{\sbt}$, it suffices to show by Proposition \ref{thm:connections} that $f^{\sbt} f_{\sbt} (x) \succeq x$ and $f_{\sbt} f^{\sbt}(y) \preceq y$.
We have
\[
  f^{\sbt} f_{\sbt} (x) = \bigvee f_{\sbt}^{-1}\left(\downarrow f_{\sbt}(x) \right).
\]
But $\setof{x} \subseteq f_{\sbt}^{-1} f_{\sbt} (x) \subseteq  f_{\sbt}^{-1}\left(\downarrow f_{\sbt}(x) \right)$ implies
\[
  f^{\sbt} f_{\sbt} (x) = \bigvee f_{\sbt}^{-1}\left(\downarrow f_{\sbt}(x) \right) \succeq \bigvee f_{\sbt}^{-1} f_{\sbt} (x)  \succeq \bigvee \setof{x} = x.
\]
For the other inequality, we have
\[
  f_{\sbt} f^{\sbt} (y) 
  = f_{\sbt} \left( \bigvee f_{\sbt}^{-1}(\downarrow y) \right) \\
  = \bigvee f_{\sbt} f_{\sbt}^{-1} (\downarrow y)
\]
since, by the first part, $f_{\sbt}$ preserves joins. Notice, $ f_{\sbt} f_{\sbt}^{-1} (\downarrow y) \subseteq~\downarrow y$. Hence,
\[
  f_{\sbt} f^{\sbt} (y) = \bigvee f_{\sbt} f_{\sbt}^{-1} (\downarrow y) \preceq \bigvee \downarrow y = y.
\]
\end{proof}
Equation~\ref{eq:adjoint-formula} is important for computational purposes as it gives an explicit formula for computing the upper adjoint of a join-preserving map.
The complexity of computing adjoints is a question of interest.

% ---------------------------------------------------------------------------------------------------
\subsection{Categories of Lattices}
\label{sec:cats}
% ---------------------------------------------------------------------------------------------------

Grandis \cite{grandis2013homological} gives several definitions of categories of lattices that we will find useful.
\begin{definition}
Let $\cat{Ltc}$ be the \define{category of lattices and connections}.
The objects, $\cat{ob}\left( \cat{Ltc}\right)$, are lattices.
Morphisms are connections,
\[
\lattice{X} \galois{f_{\sbt}}{f^{\sbt}} \lattice{Y}
\]
denoted as a pair, $f = (f_{\sbt}, f^{\sbt})$.
The identity morphism is the connection, $\lattice{X} \galois{\id}{\id} \lattice{X}$.
Composition of arrows
\[
\lattice{X} \galois{f_{\sbt}}{f^{\sbt}} \lattice{Y}  \galois{g_{\sbt}}{g^{\sbt}} \lattice{Z}
\]
is given by $g \circ f = \left( g_{\sbt} \circ f_{\sbt}, f^{\sbt} \circ g^{\sbt} \right)$.
\end{definition}

Often, we desire to work with maps that are not bi-directional in order to form limits and colimits. Consider the category, $\cat{Sup}$, of \define{lattices and join-preserving maps}. Likewise, consider the category, $\cat{Inf}$, of \define{lattices and meet-preserving maps}. We can view these categories as the images of forgetful functors,
\begin{equation}
\label{eq:forgetful}
  U:~ \cat{Ltc} \rightarrow \cat{Sup}
  \quad 
  :
  \quad
  V:~ \cat{Ltc}^\mathrm{op} \rightarrow \cat{Inf}
\end{equation}
given by the identity on objects and $U: (f_{\sbt}, f^{\sbt}) \mapsto f_{\sbt}$ and $V: (f_{\sbt}, f^{\sbt}) \mapsto f^{\sbt}$ on morphisms. For complete lattices, there are corresponding full subcategories, $\cat{cLtc}$, $\cat{cSup}$, and $\cat{cInf}$.

There are two other important subcategories of $\cat{Ltc}$: the \define{distributive} and the \define{modular} lattices,
\[
  \cat{Dlc} \subset \cat{Mlc} \subset \cat{Ltc} .
\]
A lattice, $\lattice{X}$, is \define{modular} if for $x, y \in \lattice{X}$,
\[
  x \preceq y ~\Rightarrow~ x \join ( z \meet y) = (x \join z) \meet y  
\]
for all $z \in \lattice{X}$, whereas it is \define{distributive} if for $x, y, z  \in \lattice{X}$,
\[
  x \meet (y \join z) = (x \meet y) \join (x \meet z)
\]
or dually,
\[
  x \join (y \meet z) = (x \join y) \meet (x \join z) .
\]
A distributive lattice satisfies the modular identity. A connection $f = (f_{\sbt}, f^{\sbt})$ is \define{left exact} if $f^{\sbt} f_{\sbt} (x) = x \join f^{\sbt}(0)$ and is \define{right exact} if $f_{\sbt} f^{\sbt}(y) = y \meet f_{\sbt}(1)$. An \define{exact} connection is one which is both left and right exact. The category $\cat{Mlc}$ consists of modular lattices with exact connections, while $\cat{Dlc}$ is given by distributive lattices and exact connections. Note that the subcategories $\mathsf{Mlc}$ and $\mathsf{Dlc}$ are not full subcategories of $\mathsf{Ltc}$, but $\mathsf{Dlc}$ {\em is} a full subcategory of $\mathsf{Mlc}$. 

% ---------------------------------------------------------------------------------------------------
\subsection{Cellular Sheaf Theory}
\label{sec:cellular}
% ---------------------------------------------------------------------------------------------------

Our goal is to set up and work with cellular sheaves of lattices. A sheaf is a type of data structure, built for the aggregation of local data and constraints into global solutions. The subject of sheaf theory is rich and technically intricate \cite{iversen2012cohomology, kashiwara2013sheaves,maclane2012sheaves}, but in recent years, a discrete version adapted to posets from cell complexes has been shown to be useful in a number of applications \cite{curry2014sheaves}. We therefore present a simple overview of the cellular theory.

It is assumed that the reader is familiar with the definition of a \define{cell complex}: these are slightly more general than simplicial complexes, and not quite as general as CW-complexes \cite{whitehead1949combinatorial}. A cell complex $\Space{X}$ is filtered by its $k$-skeleta $\Space{X}^{(k)}$, for $k\in\nats$, where $\Space{X}^{(0)}$ is the vertex set. We write $\Space{X}_k$ for the set of all $k$-cells in $\Space{X}$, where we identify each cell $\sigma$ with its image in $\Space{X}^{(k)}$.

Let $\cat{Fc}(\Space{X})$ denote the \define{face poset} of a cell complex $\Space{X}$ given by the transitive-reflexive closure of the relation: $\sigma \fc \tau$ if $\sigma$ is a face of $\tau$. Every cell $\sigma$ of $\Space{X}$ has:
\begin{enumerate}
% rho < sigma < tau
\item \define{boundary}, $\boundary \sigma = \setof{\rho: \rho \fc \sigma}$; and
\item \define{coboundary}, $\coboundary \sigma = \setof{\tau: \tau \fcop \sigma}$.
\end{enumerate}
It is helpful to regard the face poset as a category for what follows.

Cellular sheaves attach data to cells, glued together according to the face poset. A \define{cellular sheaf} taking values in a complete category $\cat{D}$ is simply a functor  $\mathcal{F}: \cat{Fc}(\Space{X}) \rightarrow \cat{D}$.  Explicitly, $\mathcal{F}$ attaches to each cell $\sigma$ of $\Space{X}$ an object, $\mathcal{F}_\sigma$, called the \define{stalk} over $\sigma$. For pairs $\sigma \fc \tau$, $\mathcal{F}$ prescribes \define{restriction maps}, $\resmap{F}{\sigma}{\tau}: \mathcal{F}_\sigma \rightarrow \mathcal{F}_\tau$, so that for $\rho \fc \sigma \fc \tau$,
\begin{enumerate}
\item $\resmap{F}{\sigma}{\sigma} = \id_{\mathcal{F}_{\sigma}}$
\item $\resmap{F}{\sigma}{\tau}\circ \resmap{F}{\rho}{\sigma} = \resmap{F}{\rho}{\tau}$
\end{enumerate}
The reader familiar with sheaves over topological spaces should think of a cellular sheaf as a discrete version, using the nerve of a locally-finite collection of open sets as the cell complex. 

Sheaves describe consistency or consensus relationships between data, programmed via the restriction maps: this perspective has generated applications in flow networks \cite{ghrist2013topological}, sensing \cite{ghrist2017positive}, opinion networks \cite{hansen2020opinion}, and distributed optimization \cite{hansen2019distributed}. The category of cellular sheaves over a cell complex $\Space{X}$ valued in $\cat{D}$, denoted $\cat{Sh}_\Space{X} \left( \cat{D} \right)$, has as objects sheaves, $\mathcal{F}$, and as morphisms, natural transformations $\eta:\mathcal{F}\to\mathcal{G}$. Inverse and direct images, (sometimes) tensor products, and other operations can be defined \cite{hansen2019toward, curry2014sheaves}.

In a given sheaf, the transition from local restrictions to global satisfaction is coordinated via the global section functor. The \define{(global) sections} of $\mathcal{F}$, denoted $\Gamma(\Space{X}; \mathcal{F})$, is the limit
\[
  \Gamma(\Space{X}; \mathcal{F}) = \limit \left(\mathcal{F}: \cat{Fc}(\Space{X}) \rightarrow \cat{D} \right) .
\]
There is an explicit description in terms of assignments of data to 0-cells that agree over 1-cells:
\begin{equation}
\label{eq:global-section}
%\limit \mathcal{F} = 
  \Gamma(\Space{X}; \mathcal{F}) 
  = 
  \setof{\mathbf{x} \in \prod_{v \in \Space{X}_0} \mathcal{F}_v
  ~:~
  \forall~ v \fc e \fcop w,~\resmap{F}{v}{e}(x_v) = \resmap{F}{w}{e}(x_{w})} .
\end{equation}

For cellular sheaves taking values in an abelian category, \define{sheaf cohomology} is straightforward. One forms the cochain complex $(C^\bullet, \delta^\bullet)$, where 
\[
  C^k(\Space{X}; \mathcal{F}) = \prod_{\dim\sigma=k} \mathcal{F}_\sigma ,
\]
are the $k$-cochains and the sheaf coboundary map $\delta:C^k(\Space{X}; \mathcal{F})\to C^{k+1}(\Space{X}; \mathcal{F})$ is given by
\begin{equation}
\label{eq:abelian-coboundary}
  (\delta \mathbf{x})_\tau 
  = 
  \sum_{\sigma \fc \tau} \left[ \sigma:\tau \right] \mathcal{F}_{ \sigma \fc \tau}(x_\sigma) ,
\end{equation}
where $\left[ \sigma:\tau \right] = \pm 1$ is an incidence number determined by a choice of orientation on cells of $\Space{X}$. The cohomology is then $H^k(\Space{X}; \mathcal{F})=\Ker\delta / \Im\delta$. In degree zero, this 
computes the global sections via a natural isomorphism. Cellular sheaf cohomology has proven useful in a number of settings of late \cite{coding2011ghrist,ghrist2013topological,ghrist2017positive, hansen2020opinion}, and it is our goal to extend such to the setting of sheaves of lattices.

% ---------------------------------------------------------------------------------------------------
\subsection{Cellular Hodge theory}
\label{sec:hodgecell}
% ---------------------------------------------------------------------------------------------------

The Hodge Theorem for Riemannian manifolds is a well-known example of the topological content of the Laplacian on differential forms: the kernel of the Laplacian gives the de Rham cohomology, which is isomorphic to the singular compactly supported cohomology with real coefficients \cite{warner2013foundations}. There are simple combinatorial versions of Hodge theory for cell complexes \cite{eckmann1944harmonische}, which in its most stripped-down cartoon form implies the widely-known fact that the kernel of the combinatorial graph Laplacian has dimension equal to the number of connected components of the graph \cite{chung1997spectral}.

There is a richer variant of combinatorial Hodge theory for cellular sheaves of inner-product spaces \cite{hansen2019toward}. Given a cell complex $\Space{X}$ and a cellular sheaf $\mathcal{F}$ on $\Space{X}$ taking values in the category of inner product spaces and linear transformations, one has a well-defined sheaf cohomology as per the previous subsection. The coboundary map $\delta$ from \eqref{eq:abelian-coboundary} has an adjoint $\delta^*$ given by taking the linear adjoint of each restriction map $\mathcal{F}_{ \sigma \fc \tau}$. The resulting \define{Hodge Laplacian} of $\mathcal{F}$ is given by
\begin{equation}
  \Laplacian = \delta^\ast \delta + \delta \delta^\ast.
\end{equation}
This specializes to the combinatorial graph Laplacian in the case of a constant sheaf over a cell complex. 

There is a very nice Hodge theory for sheaves of finite-dimensional real vector spaces, beginning with the observation that the kernel of $\Laplacian$ is isomorphic to the sheaf cohomology. The Hodge Laplacian is symmetric and positive semidefinite, endowing the cochain complex with a quadratic form $\langle \cdot, \Laplacian \cdot \rangle$ which, e.g., gives a measure of how close a cochain in $C^0$ is to being a global section. More is possible, including a generalization of spectral graph theory to the setting of sheaves \cite{hansen2019toward} via the spectral data of $\Laplacian$, as well as the application of diffusion dynamics via $\Laplacian$ to compute sheaf cohomology \cite{hansen2019toward}. It is the latter that we wish to generalize to sheaves of lattices.

% @@@@@@@@@@@@@@@@@@@@@@@@@@@@@@@@@@@@@@@@@@@@@@@@@@@@@@@@@@@@@@@@@@@@@@@@@@@@@@@@@@@@@@@@@@@@@@@@@@@@@@@@@@@@@@@@@
\section{The Tarski Laplacian}
\label{sec:Tarski}
% @@@@@@@@@@@@@@@@@@@@@@@@@@@@@@@@@@@@@@@@@@@@@@@@@@@@@@@@@@@@@@@@@@@@@@@@@@@@@@@@@@@@@@@@@@@@@@@@@@@@@@@@@@@@@@@@@

% ---------------------------------------------------------------------------------------------------
\subsection{Defintion and Properties}
\label{sec:defnprop}
% ---------------------------------------------------------------------------------------------------
Let $\mathcal{F}: \cat{Fc}(\Space{X}) \rightarrow \cat{Ltc}$ be a lattice-valued sheaf on a cell complex, $\Space{X}$. The restriction maps for cells $\sigma\fc\tau$ are therefore connections of the (notationally awkward) form
\[
     \resmap{F}{\sigma}{\tau} = (\resmaplower{F}{\sigma}{\tau}, \resmapupper{F}{\sigma}{\tau}). 
\]
The 0-cochains $C^0(\Space{X}; \mathcal{F})$ are choices of data on vertex stalks. There is a sensible definition of a Laplacian which mimics the Hodge Laplacian for sheaves of vector spaces. Given the role that the Tarski Fixed Point Theorem plays in what follows, it seems fitting to call this novel Laplacian by his name. 

\begin{definition}
The \define{Tarski Laplacian} for $\mathcal{F}: \cat{Fc}(\Space{X}) \rightarrow \cat{Ltc}$ is the lattice map
$\Laplacian: C^0(\Space{X}; \mathcal{F}) \rightarrow C^0(\Space{X}; \mathcal{F})$
given by
\begin{equation}
\label{eq:laplacian}
  (\Laplacian \mathbf{x})_v = \bigwedge_{e \in \coboundary v} \resmapupper{F}{v}{e} \left(\bigwedge_{w \in \partial e} \resmaplower{F}{w}{e}(x_w)   \right) 
\end{equation}
\end{definition}
This Laplacian, like the Hodge Laplacian of a cellular sheaf, defines a diffusion process in which information propagates via sheaf restriction maps.
\begin{lemma}
\label{lem:decomposition}
The Tarski Laplacian decomposes into two parts,
\begin{equation}
  (\Laplacian \mathbf{x})_v 
  = 
  \underbrace{
  \left(
    \bigwedge_{\substack{e \in \coboundary v \\ ~~~}} \resmapupper{F}{v}{e} \resmaplower{F}{v}{e}(x_v) 
  \right)
  }_{\mathrm{expanding}}  
  \meet 
  \underbrace{
  \left(
    \bigwedge_{\substack{e \in \coboundary v \\ w \in \partial e - \setof{v}}} \resmapupper{F}{v}{e} \resmaplower{F}{w}{e}(x_w)
  \right)
  }_{\mathrm{mixing}}.
\end{equation}
\end{lemma}
\begin{proof}
The map $\resmapupper{F}{v}{e} \resmaplower{F}{v}{e} \succeq \id_{\mathcal{F}_v}$ is expanding by Proposition \ref{thm:connections} while $\resmapupper{F}{v}{e}$ preserves meets by Proposition \ref{thm:aft}.
\end{proof}
Information is propagated across the 1-skeleton of $\Space{X}$ as a combination of mixing with neighboring states and expanding the local state, taking the meet of all operations. Since our lattices do not have weights, mixing and expansion are given equal priority.

\begin{lemma}
\label{thm:tarski-order}
The Tarski Laplacian $\Laplacian$ is order-preserving on the product poset $C^0(\Space{X}; \mathcal{F})$.
\end{lemma}
\begin{proof}
Suppose $\mathbf{x} \preceq \mathbf{y}$ in $C^0(\Space{X}; \mathcal{F})$. Then,
\[
  \resmaplower{F}{w}{e}(x_w) \preceq  \resmaplower{F}{w}{e}(y_w)
\]
since $\mathcal{F}_{\sbt}$ is join-preserving, hence, order preserving. This implies
\[
\bigwedge_{w \in \boundary e} \resmaplower{F}{w}{e}(x_w) 
\preceq 
\bigwedge_{w \in \boundary e} \resmaplower{F}{w}{e}(y_w)
\]
which, since $\mathcal{F}^{\sbt}$ is meet-preserving and thus order-preserving, implies
\[
  \resmapupper{F}{v}{e} 
    \left( 
      \bigwedge_{w \in \boundary e} \resmaplower{F}{w}{e}(x_w) 
    \right) 
    \preceq 
    \resmapupper{F}{v}{e} 
    \left( 
      \bigwedge_{w \in \boundary e} \resmaplower{F}{w}{e}(y_w) 
    \right) .
\]
This is turn implies
\[
  \bigwedge_{e \in \coboundary v} \resmapupper{F}{v}{e} 
    \left(
      \bigwedge_{w \in \boundary e} \resmaplower{F}{w}{e}(x_w)   
    \right) 
  \preceq 
  \bigwedge_{e \in \coboundary v} \resmapupper{F}{v}{e} 
    \left(
      \bigwedge_{w \in \boundary e} \resmaplower{F}{w}{e}(y_w)   
    \right) .
\]
Hence, $\left(\Laplacian \mathbf{x}\right)_v \preceq \left(\Laplacian \mathbf{y} \right)_v$ for every $v$.
\end{proof}

% ---------------------------------------------------------------------------------------------------
\subsection{A Fixed Point Theorem}
\label{sec:FPT}
% ---------------------------------------------------------------------------------------------------
Although the Tarski Laplacian as defined has the feel of a diffusion-type operator (Lemma~\ref{lem:decomposition}), confirmation of its fitness as a Laplacian would be welcome. We provide such in the form of a Hodge-type theorem in grading zero. 

Recall the setting of a cellular sheaf $\mathcal{F}$ of inner-product spaces on $\Space{X}$ with Hodge Laplacian $\Laplacian$~\cite{hansen2019toward}. The fact that in degree zero $\Ker\Laplacian = H^0(\Space{X}; \mathcal{F}) = \Gamma(\Space{X}; \mathcal{F})$ means that the heat equation on $C^0(\Space{X}; \mathcal{F})$,
\[
  \frac{d\mathbf{x}}{dt} = -\alpha \Laplacian \mathbf{x} 
  \quad ; \quad
  \alpha>0 ,
\]  
sends 0-cochains asymptotically to the nearest global section as $t\to+\infty$. Discretizing this in time yields a discrete-time system
\begin{equation}
\label{eq:discretetime}
  \mathbf{x}_{t+1} = (\id - \eta\Laplacian)\mathbf{x}
  \quad ; \quad
  \eta>0 .
\end{equation} 
This system likewise converges asymptotically to $\fixedpoint(\id-\Laplacian)$: harmonic 0-cochains, global sections. 

To derive a similar result on the dynamics of cochains on $\cat{Ltc}$-valued cellular sheaves, it will be necessary to use a discrete-time diffusion, given the nature of lattices. In addition, it will be necessary to forget some of the structure of the full $\cat{Ltc}$-sheaf and reduce from $\mathcal{F}$ to $\sheaf{F}=U\circ\mathcal{F}$ using the forgetful functor of \eqref{eq:forgetful}. This makes possible a fixed point description of $\Gamma(\Space{X}; \sheaf{F})$ using an analogue of \eqref{eq:discretetime}.

\begin{lemma}
\label{lem:obvious}
For $\Laplacian$ the Tarski Laplacian, $(\id\meet\Laplacian)\mathbf{x}=\mathbf{x}$ is equivalent to $\Laplacian\mathbf{x} \succeq \mathbf{x}$. 
\end{lemma}
\begin{proof}
By definition.
\end{proof}

%%%%%%%%%%%%%%%%%%%%%%%%%%%%%%%%%%%%%%%%%%%%%%%%%%%%%%%%%%%%%%%%%%%%%%%%%%%%%%%%%%%%%%%%%%%%%%%%%%%%%%%%%%%%%%%%%%
\begin{theorem}
\label{thm:main}
Let $\mathcal{F}$ be a $\cat{Ltc}$-valued cellular sheaf on $\Space{X}$ and $\Laplacian$ its Tarski Laplacian.
Then,
\begin{equation}
    \fixedpoint\left( \id \meet \Laplacian \right) = \Gamma(\Space{X};\sheaf{F}).
\end{equation}
\end{theorem}
%%%%%%%%%%%%%%%%%%%%%%%%%%%%%%%%%%%%%%%%%%%%%%%%%%%%%%%%%%%%%%%%%%%%%%%%%%%%%%%%%%%%%%%%%%%%%%%%%%%%%%%%%%%%%%%%%%
\begin{proof}
Via Lemma \ref{lem:obvious}, suppose $\mathbf{x} \in \Gamma(\Space{X}; \sheaf{F})$. Then, every term in \\ $\bigwedge_{w \in \partial e} \resmaplower{F}{w}{e}(x_w)$ is equal by the description of global sections in \eqref{eq:global-section}. By Proposition \ref{thm:connections}, $\resmapupper{F}{v}{e} \resmaplower{F}{v}{e}(x_v) \succeq x_v$ for every cobounding edge $e$.
Hence, 
\[
  (\Laplacian \mathbf{x})_v 
    = 
  \bigwedge_{e \in \coboundary v} \resmapupper{F}{v}{e} \resmaplower{F}{v}{e}(x_v) \succeq x_v,
\]
and we conclude that global sections $\mathbf{x}$ satisfy $\Laplacian \mathbf{x} \succeq \mathbf{x}$.

For the converse, suppose $(\Laplacian \mathbf{x})_v \succeq x_v$ for every $v \in \Space{X}_0$. Then, for every vertex $v$ and cobounding edge $e \in \coboundary v$,
\[
  \resmapupper{F}{v}{e} \left(\bigwedge_{w \in \boundary e} \resmaplower{F}{w}{e}(x_w) \right) 
    \succeq 
  \bigwedge_{e \in \coboundary v} \resmapupper{F}{v}{e} \left(\bigwedge_{w \in \boundary e} \resmaplower{F}{w}{e}(x_w)   \right) 
    \succeq 
  x_v.
\]
Again, by Proposition \ref{thm:connections},
\[
   \bigwedge_{w \in \boundary e} \resmaplower{F}{w}{e}(x_w) \succeq \resmaplower{F}{v}{e}(x_v) ,
\]
which in turn implies for each $w \in \boundary e$,
\begin{equation}
  \resmaplower{F}{w}{e}(x_w) \succeq \resmaplower{F}{v}{e}(x_v).
\end{equation}
Reversing the roles of $v$ and $w$ gives, via the same argument, a reversed inequality, so that
\[
  \resmaplower{F}{v}{e}(x_v)= \resmaplower{F}{w}{e}(x_w) ,
\]
proving that $\mathbf{x} \in \Gamma(\Space{X}; \sheaf{F})$.
\end{proof}

\begin{corollary}
\label{cor:Post}
For $\mathcal{F}$ as in Theorem \ref{thm:main}, $\Gamma(\Space{X}; \sheaf{F}) = \mathrm{Post}(\Laplacian)$.
\end{corollary}
\begin{proof} 
Lemma \ref{lem:obvious} combined with Theorem \ref{thm:main}.
\end{proof}

\begin{corollary}
\label{thm:sublattice}
For $\mathcal{F}$ as in Theorem \ref{thm:main}, with every vertex stalk complete, the limit
\begin{equation}
\lim \left( \cat{Fc}(\Space{X}) \xlongrightarrow{\sheaf{F}} \cat{cSup} \right) = \Gamma(\Space{X};\sheaf{F})
\end{equation}
exists and is a (nonempty) complete quasi-sublattice of $C^0(\Space{X}; \sheaf{F})$.
\end{corollary}
\begin{proof}
By Lemma \ref{thm:tarski-order}, $\Laplacian$ is order-preserving which implies $\id\meet\Laplacian$ is order-preserving.
The Tarski Fixed Point Theorem (Theorem \ref{thm:tarski}) and Theorem \ref{thm:main} complete the proof.
\end{proof}

Note that Corollary \ref{thm:sublattice} is strictly stronger than the existence of \\ $\lim \left(\sheaf{F}: \cat{Fc}(\Space{X}) \rightarrow \cat{cSup} \right)$. Since $\Gamma(\Space{X}; \sheaf{F})$ is a complete quasi-sublattice, arbitrary joins and meets of global sections exist, and, in particular, there are unique maximum and minimum global sections, even when $C^0(\Space{X}; \mathcal{F})$ is not finite.

\begin{example}
For $\Space{X}$ a 1-dimensional cell complex (a graph), the combinatorial graph Laplacian \cite{chung1997spectral} can be seen as the Hodge Laplacian of the constant sheaf of a fixed vector space. As an endmorphism on $C^0(\space{X})$, the graph Laplacian has kernel equal to the (locally) constant 0-cochains, and its dimension is the number of connected components of the graph. 

What is the Tarski analogue? Consider the constant sheaf on $\space{X}$ whose stalks are all a fixed lattice with all restriction maps the identity. The Tarski Laplacian performs a local meet with neighbors. In this case, too, the harmonic 0-cochains are precisely those which are (locally) constant. The Tarski Laplacian generalizes the graph Laplacian.
\end{example}

% ---------------------------------------------------------------------------------------------------
\subsection{Tarski Cohomology}
\label{sec:TC}
% ---------------------------------------------------------------------------------------------------

Theorem \ref{thm:main} gives an argument for the Tarski Laplacian as the ``right'' definition for $\cat{Ltc}$-valued sheaves; however, it only applies in grading zero. This is due to the difficulty of defining a natural non-abelian sheaf cohomology for $\cat{Ltc}$-valued sheaves (see \S\ref{sec:grandis}-\ref{sec:example}). Realizing the zeroth cohomology -- the global sections -- in terms of a fixed point theorem points the way to a general cohomology theory fitted to lattice-valued sheaves. 

Theorem~\ref{thm:main} and Corollary~\ref{thm:sublattice} inspire the following definitions: in what follows, $\Space{X}$ is a cell complex and $\mathcal{F}$ a $\cat{Ltc}$-valued cellular sheaf on $\Space{X}$. 

\begin{definition}
The \define{Tarski Laplacian in degree $k$} is the lattice map, $\Laplacian_k: C^k(\Space{X};\mathcal{F}) \longrightarrow C^k(\Space{X};\mathcal{F})$ acting on a $k$-cochain $\mathbf{x}$ and $k$-cell $\sigma$ via
\begin{equation}
\label{eq:laplacian-general}
  (\Laplacian_k \mathbf{x})_{\sigma} 
  = 
  \bigwedge_{\tau \in \coboundary \sigma } 
  \resmapupper{F}{\sigma}{\tau} 
  \left(
    \bigwedge_{\sigma' \in \boundary \tau} 
    \resmaplower{F}{\sigma'}{\tau}(x_{\sigma'})   
  \right) .
\end{equation}
\end{definition}
The map $\Laplacian_k$ is seen to be order-preserving by the same argument as in the proof of Lemma \ref{thm:tarski-order}.
\begin{definition}
The \define{Tarski cohomology}, $\TarH^\bullet(\Space{X};\mathcal{F})$, of a cellular sheaf $\mathcal{F}$ valued in $\cat{Ltc}$ is
\begin{equation}
  \TarH^k(\Space{X}; \mathcal{F}) 
  =
  \fixedpoint(\id\meet\Laplacian_k)
  = 
  \mathrm{Post}(\Laplacian_k) .
\end{equation}
% SINCE THERE ARE COMPETING THEORIES, PERHAPS WE SHOULD NOT DO THIS QUITE YET
% Such cochains will be called \define{harmonic}.
\end{definition}

\begin{lemma}
\label{thm:completeTH}
If $\mathcal{F}$ is valued in complete lattices and connections, $\cat{cLtc}$, then $\TarH^k(\Space{X}; \mathcal{F})$ is a (non-empty) complete quasi-sublattice of $C^k(\Space{X}; \mathcal{F})$.
\end{lemma} 
\begin{proof}
This follows immediately from the Tarski Fixed Point Theorem.
\end{proof}

% ---------------------------------------------------------------------------------------------------
\subsection{Harmonic Flow}
\label{sec:diffusion}
% ---------------------------------------------------------------------------------------------------

Fixed point theorems come in implicit [Brouwer, Lefschetz] and explicit [Banach] forms. Theorem \ref{thm:main} gives a fixed-point description of global sections which, along with higher Tarski cohomology, can be made constructive via diffusion dynamics -- using the Laplacian to define a discrete-time heat equation on cochains.   

\begin{definition}
\label{def:harmonic}
Define the \define{harmonic flow} $\Flow:\nats \times C^\bullet(\Space{X};\mathcal{F})\to C^\bullet(\Space{X};\mathcal{F})$ as $\Flow_t=\Flow(t,\cdot)=(\id\meet\Laplacian)^t$.
\end{definition}

We say that cochain $\mathbf{x}$ \textit{converges} with respect to harmonic flow, writing $\mathbf{x} \rightarrow \mathbf{x}^\ast$, if there exist $T\geq 0$ such that $\Phi_{T}(\mathbf{x}) = \Phi_{T+t}(\mathbf{x})$ for all $t \in \nats$. Then, $\mathbf{x} \rightarrow \mathbf{x}^\ast$ if and only if $\mathbf{x}^\ast \in \fixedpoint(\id\meet\Laplacian) = \mathrm{Post}(\Laplacian) = \TarH^\bullet(\Space{X}; \mathcal{F})$. 

\begin{theorem}
\label{thm:convergence}
Let $\mathcal{F}$ be a $\cat{Ltc}$-valued sheaf on a cell complex $\Space{X}$ such that (1) the number of $k$-cells is finite, and (2) the stalks over the $k$-cells satisfy the descending chain condition (DCC). Then, for some finite $t>0$, $\Flow_t$ is a projection map from $C^\bullet$ to $\TarH^\bullet$. 
\end{theorem}
\begin{proof}
Since each time-step of $\Flow$ involves a meet with $\id$, an orbit of $\Flow$ is either descending or eventually fixed. The hypotheses ensure finite termination of all initial conditions.
\end{proof}

This implies that finite or ranked stalks suffice to guarantee global finite-time convergence of the harmonic flow. Optimal bounds on the number of iterations is an interesting question; naive bounds given by heights of stalks are an exercise left to the reader. An in-depth study of the convergence properties of the Tarski Laplacian may very well imply a notion of an eigenvector/eigenvalue of a lattice endomorphism.

%
%
% @@@@@@@@@@@@@@@@@@@@@@@@@@@@@@@@@@@@@@@@@@@@@@@@@@@@@@@@@@@@@@@@@@@@@@@@@@@@@@@@@@@@@@@@@@@@@@@@@@@@@@@@@@@@@@@@@
\section{Towards Hodge Cohomology}
\label{sec:hodge}
% @@@@@@@@@@@@@@@@@@@@@@@@@@@@@@@@@@@@@@@@@@@@@@@@@@@@@@@@@@@@@@@@@@@@@@@@@@@@@@@@@@@@@@@@@@@@@@@@@@@@@@@@@@@@@@@@@

Given a cellular sheaf of lattices $\mathcal{F}$ over $\Space{X}$, it would be satisfying to have a cochain complex $(C^\bullet(X;\mathcal{F}), \coboundary)$ with sheaf cohomology from which a classical Laplacian $\Laplacian=\coboundary^*\coboundary+\coboundary\coboundary^*$ could be defined and showed isomorphic to the Tarski cohomology of \S\ref{sec:TC}. This section works to that end, following certain techniques introduced by Grandis \cite{grandis2013homological}.

% ---------------------------------------------------------------------------------------------------
\subsection{Homological Algebra of Lattices}
\label{sec:HAL}

Sheaves taking values in $\cat{Ltc}$ have enough structure to do cohomology. In $\cat{Ltc}$, the \define{zero morphism}, $(0_{\sbt}, 0^{\sbt}):\lattice{X}\to\lattice{Y}$, is the connection where $0_{\sbt}:\lattice{X}\mapsto 0$ and $0^{\sbt}:\lattice{Y}\mapsto 1$.  A morphism $f: \lattice{X} \rightarrow \lattice{Y}$ has kernel and cokernel given by equalizer and coequalizer respectively:
\[
\begin{tikzcd}
\Ker f \ar[r,"\ker f"] & \lattice{X} \ar[r, shift left= .75ex, "f"]
\ar[r, shift right = 0.75ex, swap, "0"] & \lattice{Y}
\\
 \lattice{X} \ar[r, shift left= .75ex, "f"]
\ar[r, shift right = 0.75ex, swap, "0"] & \lattice{Y} \ar[r,"\cok f"] & \Cok f
\end{tikzcd}.
\]
In $\cat{Ltc}$, one checks that these satisfy
\[
  \Ker f =~ \downset{f^{\sbt} (0)} = \setof{x \in \lattice{X}:~f_{\sbt}(x) = 0} ,
\]
\[
  \Cok f =~ \upset{f_{\sbt} (1)} = \setof{y \in \lattice{Y}:~f^{\sbt} (y) = 1} .
\]

For a morphism $f$, define its \define{normal image} by the connection, $\ker(\cok f)$, and the quasi-sublattice,
\[
\Nim f =~ \downset{f_{\sbt}(1)} .
\]
The categorial product in $\cat{Ltc}$ is the cone
\begin{equation}
\label{eq:product}
\begin{tikzcd}
  \lattice{X}  & \lattice{X} \prod \lattice{Y} \ar[l, "p",swap] \ar[r, "q"] & \lattice{Y} ,
\end{tikzcd}
\end{equation}
where $\lattice{X} \prod \lattice{Y}$ is the cartesian-product lattice, and where $p_{\sbt}(x,y) = x$, $p^{\sbt}(x)= (x, 1)$, and $q_{\sbt}(x,y) = y$, $q^{\sbt} (y) = (1, y)$.
The coproduct is the cocone
\begin{equation}
\begin{tikzcd}
  \lattice{X} \ar[r, "i"] & \lattice{X} \coprod \lattice{Y}   & \lattice{Y} \ar[l, "j",swap] ,
\end{tikzcd}
\end{equation}
where $\lattice{X} \coprod \lattice{Y}$ is again the cartesian-product lattice, and where $i_{\sbt}(x) = (x, 0)$, $i^{\sbt}(x,y) = x$, and $j_{\sbt}(y) = (0, y)$, $j^{\sbt}(x,y) = y$.

Recall that a \define{semi-additive category} is a pointed category with a \emph{biproduct}, denoted $\times$: a special product that is compatible with the coproduct and coincides with the coproduct on objects. The category $\cat{Ltc}$ is semi-additive; compatibility is satisfied, as $p \circ i = \id$ and $q \circ j = \id$, as well as $p \circ j = 0$ and $q \circ i = 0$. The objects $\lattice{X} \prod \lattice{Y}$ and $\lattice{X} \coprod \lattice{Y}$ coincide by definition.

\begin{lemma}
\label{lem:semi-additive}
Semi-additive categories are enriched in abelian monoids.
\end{lemma}
\begin{proof}
The diagonal ($\diagonal$) and codiagonal ($\codiagonal$) morphisms come out of the product and coproduct respectively. For $f, g \in \hom(\lattice{X}, \lattice{Y})$, $f+g$ is the composition in the diagram
\begin{equation}
\begin{tikzcd}
\lattice{X} \ar[r,"\diagonal"] \ar[rrr,out=-30,in=210,swap,"f+g"] & \lattice{X} \times \lattice{X} \ar[r, "f \times g"] & \lattice{Y} \times \lattice{Y} \ar[r, "\codiagonal"] & \lattice{Y}
\end{tikzcd}.
\end{equation}
We leave the details of showing $f + g = g + f$ and $f + 0 = f$ as an exercise to the reader.
\end{proof}
\begin{corollary}
\label{thm:semi-additive}
In $\cat{Ltc}$, $\diagonal$ is the connection, $\diagonal_{\sbt}(x) = (x, x)$, $\diagonal^{\sbt}(x,x') = x \meet x'$, and $\nabla$ is the connection, $\codiagonal_{\sbt} (y, y') = y \join y'$, $\codiagonal^{\sbt} (y) = (y, y)$.
Furthermore, $f + g$ is the following connection:
\[
(f+g)_{\sbt} (x) = f_{\sbt}(x) \join g_{\sbt}(x) ,
\]
\[
(f + g)^{\sbt} (y) = f^{\sbt}(y) \meet g^{\sbt} (y) .
\]
\end{corollary}
\noindent We denote sums of connections by $f + g$ and $\sum_\alpha f_\alpha$ as appropriate.
% ---------------------------------------------------------------------------------------------------
\subsection{Grandis Cohomology}
\label{sec:grandis}
% ---------------------------------------------------------------------------------------------------
With this in place, one can examine the category of bounded cochain complexes $(C^\bullet, \coboundary)$ valued in $\cat{Ltc}$, denoted $\cat{Ch}_{+}(\cat{Ltc})$, and define a cohomology. This is implicit in the work of Grandis \cite{grandis2013homological}; we fill in certain details for clarity.

As with ordinary cohomology, define the \define{cocycles} and \define{coboundaries} of a cochain complex valued in $\cat{Ltc}$ by $Z^\bullet = \Ker \coboundary$ and $B^\bullet = \Nim \coboundary$ respectively. (For simplicity, we omit the grading superscript on the coboundary operator; in the remainder of this section, the reader should be careful when specifying to a particular grading).

\begin{lemma}
\label{thm:boundary-inclusion}
There is an order inclusion of coboundaries into cocycles.
\end{lemma}
\begin{proof}
Suppose $x$ is a boundary, so that $x \preceq  \coboundary_{\sbt}(1)$ and $\coboundary_{\sbt}(x) \preceq \coboundary_{\sbt}  \coboundary_{\sbt}(1) = 0$ as $\coboundary\circ\coboundary = 0$.
\end{proof}

\begin{definition}
The \define{Grandis cohomology} of $(C^\bullet, \coboundary) \in \cat{Ch}_{+}(\cat{Ltc})$ is
\begin{equation}
\label{eq:ltc-cohomology}
  \GraH^\bullet(C^\bullet) = \Cok ( B^\bullet \hookrightarrow Z^\bullet) .
\end{equation}
\end{definition}

Grandis cohomology is best computed as an interval.

\begin{proposition}
\label{thm:interval}
$\GraH^\bullet$ is isomorphic to
\[
  \left[ \coboundary_{\sbt}(1),~\coboundary^{\sbt}(0) \right] .
\]
\end{proposition}
\begin{proof}
Let $x \in \Cok ( B^\bullet \hookrightarrow Z^\bullet)$ so that $x \succeq \coboundary_{\sbt}(1)$. Simultaneously, $x \preceq \coboundary^{\sbt}(0)$ as $x \in \Ker \coboundary^k$.
\end{proof}

% ---------------------------------------------------------------------------------------------------
\subsection{Passing from Vector Spaces to Subspace Lattices}
\label{sec:example}
% ---------------------------------------------------------------------------------------------------

The problem with defining a sheaf cohomology for sheaves valued in $\cat{Ltc}$ lies in the definition of the coboundary map: for ordinary sheaf cohomology of sheaves valued in $\cat{Vect}$, the abelian structure is vital to defining $\coboundary$.  However, if one begins with a sheaf valued in $\cat{Vect}$, then it is possible to pass to $\cat{Ltc}$ via a Grassmannian --- converting all vector space data to the lattice of subspaces. This is an interesting class of sheaves and, though not universal, does provide an arena in which to compare different sheaf cohomologies. 

Denote by $\cat{Gr}$ the \define{transfer functor}, $\cat{Gr}(\cdot): \cat{Vec} \rightarrow \cat{Mlc}$, which converts vector spaces and linear transformations to (modular) lattices of subspaces and exact connections, where, for a linear transformation $A$, $Gr(A)_{\sbt}$ is the image and $Gr(f)^{\sbt}$ is the inverse image. By abuse of notation, this extends to chain complexes in the appropriate categories as well. This transfer functor respects cohomology.

\begin{theorem}
\label{thm:grandis-cohom}
For $C^\bullet$ a cochain complex valued in $\cat{Vect}$, 
\begin{equation}
  \GraH^\bullet \left( \cat{Gr} (C^\bullet) \right) \cong \cat{Gr} \left( H^\bullet(C^\bullet)\right) .
\end{equation}
\end{theorem}

\begin{proof}
That $\cat{Gr}\left( \cdot \right)$ is a functor is left as an exercise. It is clear that a connection induced by the coboundary maps in $C^\bullet$ is (1) exact and (2) a coboundary map in $\cat{Gr}(C^\bullet)$. By Proposition \ref{thm:interval},
\begin{align*}
\GraH^k \left( \cat{Gr} (C^\bullet) \right)	
											& \cong \left[ \cat{Gr}(\coboundary)_{\sbt} (1), \cat{Gr} (\coboundary)^{\sbt} (0) \right] \\
											& = \left[ \Im \coboundary, \Ker \coboundary \right] .
\end{align*}
Then, by the Fourth Isomorphism Theorem \cite{dummit2004abstract}[p.~394],
\[
\left[ \Im \coboundary, \Ker \coboundary \right] \cong \cat{Gr} \left( \Ker \coboundary / \Im \coboundary \right) .
\]
\end{proof}

In the case of a sheaf $\cat{F}$ of vector spaces over $\Space{X}$, we can pass to $\mathcal{F}=\cat{Gr}(\cat{F})$, the induced sheaf taking values in $\cat{Mlc}$ via subspaces. 
\begin{corollary}
\label{thm:grandis-sheaf}
For $\mathcal{F}=\cat{Gr}(\cat{F})$ a sheaf of lattices induced by a sheaf of vector spaces, $\GraH^\bullet\left(\Space{X}; \mathcal{F} \right) \cong \cat{Gr}\left(H^\bullet(\Space{X}; \cat{F})\right)$.
\end{corollary}
For such sheaves, the relationship between the Tarski-based and Grandis-based cohomologies is as follows:
\begin{proposition}
For sheaves of the form $\mathcal{F} = \cat{Gr}(\cat{F})$, the zeroth Grandis cohomology is a quasi-sublattice of the zeroth Tarski cohomology.
\end{proposition}
\begin{proof}
If $\mathbf{x}$ is a subspace in $\GraH^0(\Space{X}; \mathcal{F})$, then by Corollary \ref{thm:grandis-sheaf}, $\mathbf{x}$ is a subspace of $\Gamma(\Space{X}; \cat{F}) = H^0(\Space{X}; \cat{F})$. In turn, this means that the images of $\mathbf{x}$ under the restriction maps $\cat{F}_{v \fc e}$, $\cat{F}_{w \fc e}$ will coincide for every $v, w \in \boundary e$ so that $\mathbf{x} \in \Gamma(\Space{X}; \sheaf{F}) = \TarH^0(\Space{X}; \mathcal{F})$.
\end{proof}

The inclusion is strict. Consider a twisted sheaf $\cat{F}$ of 1-dimensional stalks as in Fig.\ \ref{fig:counter-eg}. Then, the lattice-valued sheaf $\mathcal{F}=\cat{Gr}(\cat{F})$ is constant: the stalks are the lattice $\lattice{L} = \setof{0, 1}$, and the restriction maps are all the identity. Hence, $\TarH^0(\Space{X}; \mathcal{F}) \cong \lattice{L}$. However (as the reader may calculate), $H^0(\Space{X}; \cat{F}) \cong 0$, as there are no nonzero global sections. By Theorem \ref{thm:grandis-cohom}, $\GraH^0(\Space{X}; \mathcal{F}) \cong \lattice{0}$.

\begin{figure}[h]
%\begin{tikzpicture}[baseline= (a).base]
%\node[scale=0.8] (a) at (0,0){
%\begin{tikzcd}
%				\R \ar[rd,swap, "-\id"] \ar[dd,"\id"]	&		&		\\			
%		    		&	\R	&		\\
%		  	 	\R	&		&	 \R \ar[ul,"\id"] \ar[ld,"\id"]\\
%		     		&	\R	&		\\ 
%		    	\R \ar[ur,swap,"\id"] \ar[uu,swap,"\id"]	&		&		\\  	
%\end{tikzcd}
%};
%\end{tikzpicture}
\includegraphics[width=\textwidth]{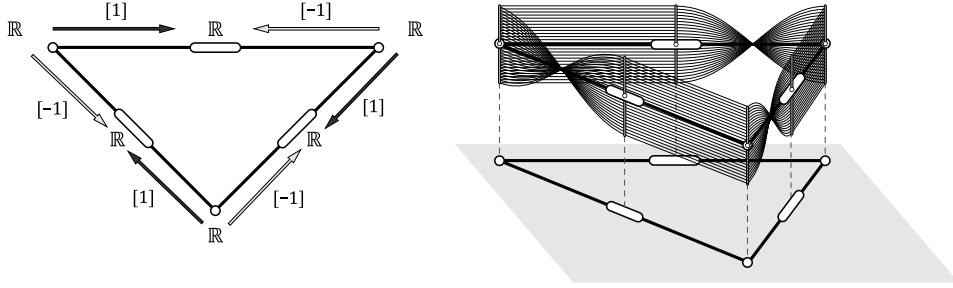}
\caption{A twisted sheaf (left) demonstrates strict inclusion $\TarH^0 \nsubseteq \GraH^0$ due to the absence of a nonzero section (right).}
\label{fig:counter-eg}
\end{figure}

% ---------------------------------------------------------------------------------------------------
\subsection{A Hodge Complex}
\label{sec:HC}
% ---------------------------------------------------------------------------------------------------

Sheaves of lattices do not generically factor through the transfer functor. For such sheaves, can their cohomology be read off of a cochain complex via Grandis cohomology? This is unclear if not unlikely.
We therefore pursue a philosophically disparate lattice sheaf cohomology mimicking the Hodge theory for vector-valued sheaves.
We define a pseudo-coboundary connection $\qcoboundary : C^\bullet(\Space{X}; \mathcal{F}) \rightarrow C^{\bullet+1}(\Space{X}; \mathcal{F})$ given by a sum of projections composed with restriction maps: sums naturally arise from the semi-additive structure on $\cat{Ltc}$; projections onto a stalk ($\mathcal{F}_\sigma$) over a $k$-cell ($\sigma \in \Space{X}_k$) are connections,
\[
  \pi_{\sigma}: C^k(\Space{X}; \mathcal{F}) \rightarrow \mathcal{F}_{\sigma}:~\quad (\pi_{\sigma})_{\sbt}(\mathbf{x}) = x_\sigma,~\quad 
(\pi_{\sigma})^{\sbt}(x_\sigma) = (1, \dots,1, x_\sigma, 1, \dots, 1).
\]

\begin{definition}
Let $\mathcal{F}$ be a sheaf valued in $\cat{Ltc}$ with cochains $C^\bullet(\Space{X}; \mathcal{F})$ per usual. The \define{pseudo-cochain connections} of $\mathcal{F}$ are the sequence of connections
\[
  \qcoboundary : C^\bullet(\Space{X}; \mathcal{F}) \rightarrow C^{\bullet+1}(\Space{X}; \mathcal{F})
\] 
given by
\begin{equation}
\label{eq:qcoboundary}
  (\qcoboundary \mathbf{x})_\tau 
  = 
  \left( 
    \sum_{\sigma \fc \tau} \resmap{F}{\sigma}{\tau} \circ \pi_\sigma 
  \right) 
  \left(  
    \mathbf{x} 
  \right) .
\end{equation}
The \define{pseudo-cochain comlpex} of $\mathcal{F}$ is the lattice of cochains $C^\bullet(\Space{X}; \mathcal{F})$ together with the pseudo-cochain connections of $\mathcal{F}$.
\end{definition}

The connection $\qcoboundary$ is not a true coboundary as $\qcoboundary\circ\qcoboundary\neq 0$ in general. As a consequence, a lattice of coboundaries is are not a quasi-sublattice of cycles (Lemma \ref{thm:boundary-inclusion}) and Grandis cohomology is not well-defined. However, as connections come with adjoints built-in, a Hodge-like Laplacian may still be defined.

\begin{definition}
\label{def:hodge-ltc}
For $(C^\bullet, \coboundary)$ a (bounded) cochain complex in $\cat{Ltc}$, the \define{Hodge Laplacians} are the pair $\Laplacian^{+}, \Laplacian^{-}: C^\bullet \rightarrow C^\bullet$ of order-preserving degree-zero maps
\begin{equation}
\label{eq:updown}
  \Laplacian^{+}_k = \coboundary^{\sbt} \coboundary_{\sbt} 
  \quad : \quad
  \Laplacian^{-}_k = \coboundary_{\sbt} \coboundary^{\sbt} . 
\end{equation}
As with the cellular (combinatorial) Hodge Laplacian, one calls $\Laplacian^+$ the \define{up-Laplacian} and $\Laplacian^-$ the \define{down-Laplacian}.
\end{definition}

For example, the Hodge Laplacians of sheaves factoring through the transfer functor are readily calculated:
\begin{proposition}
For $\mathcal{F} = \cat{Gr}(\cat{F})$ and $\cat{Gr}(C^\bullet, \coboundary)$ its sheaf cochain complex, the Hodge Laplacians are computed via
\[
  \Laplacian^{+} \mathbf{x} = \mathbf{x} \join \Ker \coboundary
  \quad : \quad
  \Laplacian^{-} \mathbf{x} = \mathbf{x} \meet \Im \coboundary .
\]
\end{proposition}
\begin{proof}
As $\cat{Gr}(\coboundary)$ is an exact connection, 
\[
  \coboundary^{\sbt} \coboundary_{\sbt} (\mathbf{x}) = \mathbf{x} \join \coboundary^{\sbt}(0)
  \quad : \quad
  \coboundary_{\sbt} \coboundary^{\sbt} (\mathbf{x}) = \mathbf{x} \meet \coboundary_{\sbt}(1) .
\]
Here, $\coboundary^{\sbt} 0$ is the preimage of the zero vector space (the kernel) and $\coboundary_{\sbt} 1$ is the direct image of the entire space (the image).
\end{proof}

More generally, we may compute the Hodge Laplacians of the pseudo-cochain complex for an arbitrary lattice-valued sheaf. By slight abuse of notation, substitute the pseudo-coboundary $\qcoboundary$ of \eqref{eq:qcoboundary} into the definition of the up- and down-Laplacians of \eqref{eq:updown}.
\begin{proposition}
Let $\mathcal{F}$ be a sheaf valued in $\cat{Ltc}$ and $(C^\bullet, \qcoboundary)$ its pseudo-cochain complex. Then, the Hodge Laplacians ($\Laplacian^{+} = \qcoboundary^{\sbt} \qcoboundary_{\sbt}$, $\Laplacian^{-} = \qcoboundary_{\sbt} \qcoboundary^{\sbt}$) are computed as
\begin{equation}
\label{eq:hodge-up}
  (\Laplacian^{+} \mathbf{x})_\sigma 
  = 
  \bigmeet_{\tau \in \coboundary \sigma} 
    \resmapupper{F}{\sigma}{\tau} 
    \left( 
      \bigjoin_{\sigma' \in \boundary \tau} \resmaplower{F}{\sigma'}{\tau}(x_{\sigma'})
    \right) ,
\end{equation}
\begin{equation}
\label{eq:hodege-down}
  (\Laplacian^{-} \mathbf{x})_\sigma 
  = 
  \bigjoin_{\rho \in \boundary \sigma} 
  \resmaplower{F}{\rho}{\sigma} 
  \left( 
    \bigmeet_{\sigma' \in \coboundary \rho} \resmapupper{F}{\rho}{\sigma'}(x_{\sigma'}) 
  \right) .
\end{equation}
\end{proposition}
\begin{proof}
The definitions, Corollary \ref{thm:semi-additive}, and a few minor details suffice.
\end{proof}

A return to the fixed point perspectives of \S\ref{sec:Tarski} suggests yet another candidate for $\cat{Ltc}$-valued sheaf cohomology.

\begin{definition}
For the pseudo-cochain complex of $\mathcal{F}$ as above, the \define{upper-} and \define{lower-Hodge cohomology} lattices of $\mathcal{F}$ are
\begin{equation}
  \HodHup^\bullet(\Space{X}; \mathcal{F}) = \mathrm{Post}(\Laplacian^{+}),
\end{equation}
\begin{equation}
  \HodHdown^\bullet(\Space{X}; \mathcal{F}) = \mathrm{Pre}(\Laplacian^{-}).
\end{equation}
\end{definition}

As a compression of cochains, this Hodge cohomology is less efficient than Tarski cohomology.
\begin{proposition}
Tarski cohomology is a quasi-sublattice of upper-Hodge cohomology.
\end{proposition}
\begin{proof}
Suppose $\mathbf{x} \in \TarH^k(\Space{X}; \mathcal{F})$ and $\sigma\fc\tau\fcop\sigma'$ is a pair of $k$-cell faces of a common $(k+1)$-cell $\tau$. Then, following the proof of Theorem \ref{thm:main}, one has
%for every $\sigma \in \Space{X}_k$, $\tau \in \coboundary \tau$, $\sigma' \in \boundary \tau$,
\[
  \resmaplower{F}{\sigma}{\tau} (x_\sigma) 
  \preceq 
  \bigmeet_{\sigma' \in \boundary \tau} 
    \resmaplower{F}{\sigma'}{\tau} (x_{\sigma'}) 
  \preceq 
    \resmaplower{F}{\sigma'}{\tau} (x_{\sigma'}).
\]
By symmetry of $\sigma\fc\tau\fcop\sigma'$, this implies $\resmaplower{F}{\sigma}{\tau}(x_\sigma) = \resmaplower{F}{\sigma'}{\tau}(x_{\sigma'})$. Then,
\[
  \bigmeet_{\tau \in \coboundary \sigma} 
  \resmapupper{F}{\sigma}{\tau} 
  \left( 
    \bigjoin_{\sigma' \in \boundary \tau} 
    \resmaplower{F}{\sigma'}{\tau} (x_{\sigma'})
  \right) 
  = 
  \bigmeet_{\tau \in \coboundary \sigma} 
  \resmapupper{F}{\sigma}{\tau} 
  \resmaplower{F}{\sigma}{\tau} 
  (x_{\sigma}) 
  \succeq 
  x_\sigma
\]
so that $\mathbf{x} \in \HodHup^k(\Space{X};\mathcal{F})$.
\end{proof}

For an explicit example where Tarski cohomology is strictly contained in upper-Hodge cohomology, consider a constant sheaf on $\Space{X}$, a connected graph with three vertices ($u, v, w$) and two edges ($u \sim v$, $v \sim w$). The Tarski Laplacian is the endomorphism
\[
\Laplacian_0 \mathbf{x} = \left(x_u \meet x_v, x_u \meet x_v \meet x_w, x_v \meet x_w \right);
\]
the upper-Hodge laplacian is the endomorphism
\[
\Laplacian_0^{+} \mathbf{x} = \left( x_u \join x_v, (x_u \join x_v) \meet (x_v \join x_w), x_v \join x_w \right).
\]
$\TarH^0$ is the lattice of constant sections.
To see that $\TarH^0 \subsetneq \HodHup^0$, consider cochain $\mathbf{x} = (x, y, x)$ such that $y \prec x$. Then, $\Laplacian_0^{+} \mathbf{x} = (x, x, x) \succeq \mathbf{x}$. Thus, $\mathbf{x} \in \mathrm{Post}(\Laplacian_0^{+}) = \HodHup^0$, but is not a constant section.

% @@@@@@@@@@@@@@@@@@@@@@@@@@@@@@@@@@@@@@@@@@@@@@@@@@@@@@@@@@@@@@@@@@@@@@@@@@@@@@@@@@@@@@@@@@@@@@@@@@@@@@@@@@@@@@@@@
\section{Summary}
\label{sec:summary}

We close with Table \ref{table:1}, summarizing our constructions for cellular sheaves of lattices. The Grandis cohomology -- which is defined only when there is a cochain complex, such as in the case of factoring through $\cat{Vec}$ --- is the ``smallest'' cohomology, followed by the Tarski and then the Hodge theories:
\begin{equation}
  \GraH^\bullet \subset \TarH^\bullet \subset \HodHup^\bullet .
\end{equation} 

%-----------------------------------------------------------------------------------------------------
\begin{table}[h]
\resizebox{\textwidth}{!}{\begin{tabular}{ | c || c | c | c | c | c|} 
\hline
	& symb.\ & target\ & coboundary, $(\coboundary_k \mathbf{x})_\tau$ & cohomology & Laplacian, $(\Laplacian_k^{\pm} \mathbf{x})_{\sigma}$ \\
	\hline
	Cellular & 	$H^k$	&  $\cat{Hilb}$ & $\sum_{\sigma \fc \tau} \left[ \sigma:\tau \right] \mathcal{F}_{ \sigma \fc \tau}(x_\sigma)$ & $\Ker \coboundary_k / \Im \coboundary_{k-1}$ & $\left( \delta^\ast_k \delta_k + \delta_{k-1} \delta^\ast_{k-1} \right)_{\sigma}$ \\
	\hline
	Tarski & $\TarH^k$ & $\cat{Ltc}$ &  & $\mathrm{Post}(\Laplacian_k \mathbf{x})$ & $\bigwedge_{\tau \in \coboundary \sigma }  \resmapupper{F}{\sigma}{\tau}\left(\bigwedge_{\sigma' \in \boundary \tau} \resmaplower{F}{\sigma'}{\tau}(x_{\sigma'}) \right)$ \\
	\hline
	Grandis & $\GraH^k$ & $\cat{Mlc}$* & $\cat{Gr}(\coboundary)$ & $\cat{Gr}\left(H^k(\Space{X}; \cat{F})\right)$ & $x_\sigma \join (\Ker \coboundary_k)_\sigma$  \\
			&			&			&			& 	&$x_\sigma \meet (\Im \coboundary_{k-1})_\sigma$ \\
	\hline
	Hodge 	& $\HodHup^k$ & $\cat{Ltc}$ & $\left( \sum_{\sigma \fc \tau} \resmap{F}{\sigma}{\tau} \circ \pi_\sigma \right) \left( \mathbf{x} \right)$ & $\mathrm{Post}(\Laplacian_k^{+})$ & $\bigmeet_{\tau \in \coboundary \sigma} \resmapupper{F}{\sigma}{\tau} \left( \bigjoin_{\sigma' \in \boundary \tau} \resmaplower{F}{\sigma'}{\tau}(x_{\sigma'})\right)$ \\
				& $\HodHdown^k$ &  & & $\mathrm{Pre}(\Laplacian_k^{-})$ & $\bigjoin_{\rho \in \boundary \sigma} \resmaplower{F}{\rho}{\sigma} \left( \bigmeet_{\sigma' \in \coboundary \rho} \resmapupper{F}{\rho}{\sigma'}(x_{\sigma'}) \right)$ \\
	\hline
\end{tabular}}
\caption{Cohomologies and their Laplacians for cellular sheaves of lattices. \newline \footnotesize  *Grandis cohomology is defined for sheaves valued in $\cat{Mlc}$ whenever $\cat{Mlc}$ factors through $\cat{Vec}$.}
\label{table:1}
\end{table}

% @@@@@@@@@@@@@@@@@@@@@@@@@@@@@@@@@@@@@@@@@@@@@@@@@@@@@@@@@@@@@@@@@@@@@@@@@@@@@@@@@@@@@@@@@@@@@@@@@@@@@@@@@@@@@@@@@
% @@@@@@@@@@@@@@@@@@@@@@@@@@@@@@@@@@@@@@@@@@@@@@@@@@@@@@@@@@@@@@@@@@@@@@@@@@@@@@@@@@@@@@@@@@@@@@@@@@@@@@@@@@@@@@@@@
% BIBLIO

\end{document}